\documentclass[10pt,a4paper]{amsart}
\usepackage[latin1]{inputenc}
\usepackage{amsmath}
\usepackage{amsfonts}
\usepackage[T1]{fontenc}
\usepackage{amssymb}
\usepackage{graphicx}
\usepackage{amscd}
\usepackage{mathrsfs}
\usepackage[all]{xy}
\usepackage{overpic}

\usepackage{slashed} 

\usepackage[T1]{fontenc}
\newtheorem{theorem}{Theorem}
\usepackage{amsmath}

\newtheorem{thm}{Theorem}[section]
\newtheorem{lemma}[thm]{Lemma}
\newtheorem{prop}[thm]{Proposition}

\theoremstyle{definition}

\newtheorem{defi}[thm]{Definition}

\theoremstyle{remark}
\newtheorem{remark}[thm]{Remark}
\numberwithin{equation}{section}

\newcommand{\ad}{\text{ad}}

\newcommand{\tr}{\text{tr}}                        

\newcommand{\Om}{\Omega}

\newcommand{\Real}{\mathbb{R}}       

\DeclareMathOperator{\rk}{rank}

\newcommand{\RR}{{\mathbb R}}

\newcommand{\surj}{\to\kern-1.8ex\to}

\newcommand{\lra}[1]{\stackrel{#1}{\longrightarrow}}

\newcommand{\cA}{\mathcal{A}}
\newcommand{\cC}{\mathcal{C}}

\newcommand{\cE}{\mathcal{E}}

\newcommand{\cM}{\mathcal{M}}
\newcommand{\cP}{\mathcal{P}}

\newcommand{\cG}{\mathcal{G}}

\newcommand{\cR}{\mathcal{R}}

\newcommand{\Lie}{\operatorname{Lie}}

\newcommand{\cH}{\mathcal{H}} 

\def\om{\omega}
\def\Om{\Omega}

\def\Lie{\mathrm{Lie}}
\def\Diff{\mathrm{Diff}}
\def\Aut{{\mathrm{Aut}}}

\def\Im{\mathrm{Im}}

\def\ctG{\widetilde{\mathcal{G}}}

\def\cA{\mathcal{A}}

\def\cG{\mathcal{G}}
\def\cH{\mathcal{H}}
\def\cP{\mathcal{P}}

\def\cE{\mathcal{E}}
\def\cR{\mathcal{R}}

\def\bfL{\mathbf{L}}

\def\bfP{\mathbf{P}}

\def\hbfL{\hat{\mathbf{L}}}
\def\hbfP{\hat{\mathbf{P}}}

\newcommand{\ra}{\rangle}
\newcommand{\la}{\langle}
\newcommand{\ot}{\otimes}

\begin{document}

\title[$G2$, Strominger system and Moduli]{Moduli of $G_2$ structures and the Strominger system in dimension $7$}

\author[A. Clarke]{Andrew Clarke}
\address{Instituto de Matem\'atica, Universidade Federal do Rio de Janeiro,
Av. Athos da Silveira Ramos 149,
Rio de Janeiro, RJ, 21941-909,
Brazil.}
\email{andrew@im.ufrj.br}

\author[M. Garcia-Fernandez]{Mario Garcia-Fernandez}
\address{Instituto de Ciencias Matem\'aticas (CSIC-UAM-UC3M-UCM)\\
  Nicol\'as Cabrera 13--15, Cantoblanco\\ 28049 Madrid, Spain}
  \email{mario.garcia@icmat.es}
\author[C. Tipler]{Carl Tipler}
\address{D\'epartement de math\'ematiques,
Universit\'e de Bretagne Occidentale,
6 Av. Victor Le Gorgeu, 29238, Brest Cedex 3, France.
}
\email{carl.tipler@univ-brest.fr}

\date{\today}

\maketitle


\begin{abstract}

We consider $G_2$ structures with torsion coupled with $G_2$-instantons, on a compact $7$-dimensional manifold. The coupling is via an equation for $4$-forms which appears in supergravity and generalized geometry, known as \emph{the Bianchi identity}. The resulting system of partial differential equations can be regarded as an analogue of the Strominger system in $7$-dimensions. We initiate the study of the moduli space of solutions and show that it is finite dimensional using elliptic operator theory. We also relate the associated geometric structures to generalized geometry. 

\end{abstract}

\section{Introduction}\label{sec:intro}
The study of the moduli space of torsion-free $G_2$ structures on a compact $7$-dimensional spin manifold $M$ was initiated by
Joyce \cite{joy0,joy}, who proved that the \emph{period map} defines a local diffeomorphisms 
to the cohomology group $H^3(M,\mathbb{R})$. Further, there is a natural pseudo-Riemannian metric of Hessian type on the 
moduli space first defined by Hitchin \cite{Hit0}. 
In this work we address the question of whether $G_2$ structures with torsion on a compact $7$-dimensional manifold arise in 
moduli and, if so, what is the nature of their moduli space. Our approach to this problem is inspired by physics, 
and relies on recent developments on the study of the Strominger system of partial differential equations \cite{grt}. 
Here, we consider $G_2$ structures  with torsion (cocalibrated, of type $W3$) coupled with $G_2$ instantons, 
by means of an equation for $4$-forms which appears in supergravity and generalized geometry \cite{Strom,Hit1}. 
The resulting system of equations can be regarded as an analogue of the Strominger system in $7$-dimensions. 
Our main result implies the finite-dimensionality of the moduli space.
As a consequence of our method, we also obtain a new proof of the characterization of the infinitesimal moduli space of $G_2$-holonomy metrics in terms of $H^3(M,\mathbb{R})$, as originally observed by Joyce and Hitchin.

Geometry in $6$ and $7$ dimensions are closely tied by a basic linear algebra fact: the existence of \emph{non-degenerate} real $3$-forms \cite{Hit0}. A single non-degenerate $3$-form on a manifold $M$ defines a reduction of the structure group of $M$ to $SL(3,\mathbb{C})$ if $\dim M = 6$ --- and hence a Calabi-Yau structure ---, and a reduction to $G_2$ if $\dim M = 7$. 

An important part of the classical theory of Calabi-Yau three-folds has evolved recently to the study of $SU(3)$-structures with torsion, and in particular towards the understading of a complicated system of partial differential equations motivated by physics, known as the Strominger system \cite{Strom}. The mathematical study of the Strominger system (see \cite{TsengYau} for a survey) was initiated by Li and Yau as a natural generalization of the Calabi problem, and in relation to `Reid's fantasy', on the moduli space of complex $3$-folds with trivial canonical bundle and varying topology. Relying on the common geometric features of $6$ and $7$ dimensions, and the recent progress made in the understanding
of the moduli space for the Strominger system \cite{grt}, it is natural to ask whether there is a similar pattern in $7$-dimensions which helps to shed light on the moduli space of 
$G_2$-structures with torsion.

From the point of view of physics, the Strominger system is a particular instance of a more general system of equations, known as \emph{the Killing spinor equations in (heterotic) supergravity}. The \emph{compactification} of the physical theory leads to the study 
of models of the form $N^k\times M^{10-k}$, where $N^k$ is a $k$-dimensional Lorentzian manifold 
and $M^{10-k}$ is a Riemannian spin manifold which encodes the extra dimensions of a supersymmetric vacuum. The Killing spinor equations, for a Riemannian metric $g$, a spinor $\Psi$, a function $\phi$ (the dilaton) and a $3$-form $H$ (the NS-flux) on $M^{10-k}$, can be written as
\begin{equation}
 \label{eq:Killing spinors 10dim}
 \nabla \Psi=0, \qquad \qquad (d\phi - \frac{1}{4}H)\cdot \Psi=0
\end{equation}
where $\nabla$ is a $g$-compatible connection with skew-symmetric torsion $H$. 
Solutions to (\ref{eq:Killing spinors 10dim}) provide rich geometrical structures on $M$. If the
torsion $H$ vanishes, the existence of a parallel spinor reduces the holonomy of the Levi-Civita connection on $M$ to $SU(n), Sp(n),G_2$ or $Spin(7)$ according to its dimension.
However, the torsion-free condition, often equivalent to the condition $dH=0$ --- the so-called strong solutions ---, is very restrictive, as many interesting solutions to the equations arise in   
manifolds equiped 
with metric connections with skew-symmetric torsion and holonomy contained in $SU(n), Sp(n),G_2$ or $Spin(7)$.

An interesting relaxation of the notion of strong solution is provided by the \emph{Bianchi identity} (related to the anomaly cancellation condition in string theory), which requires a correction of $dH$ of the form
\begin{equation}
 \label{eq:Anomaly}
 dH= \tr(F_A\wedge F_A)-\tr(R_\nabla\wedge R_\nabla)
\end{equation}
where $F_A$ is the curvature of a connection $A$ on a principal $K$-bundle $P_K$ over $M$
and $R_\nabla$ is the curvature of an additional linear connection $\nabla$ on the tangent bundle of $M$. 
The extra requirements for a solution of the Killing spinor equations \eqref{eq:Killing spinors 10dim} and the Bianchi identity \eqref{eq:Anomaly} to provide with a supermetric vacuum of the theory is given by the instanton conditions \cite{Ivan09}
\begin{equation}
 \label{eq:instanton}
R_\nabla \cdot \Psi =0, \qquad \qquad   F_A\cdot \Psi =0.
\end{equation}
In a $6$-dimensional compact manifold $M$, the combination of the above mentioned equations (\ref{eq:Killing spinors 10dim}), (\ref{eq:Anomaly}) and (\ref{eq:instanton}) leads to the Strominger system. In this paper,
we initiate the study of the moduli space of solutions to the Strominger equations in $7$-dimensions, that we introduce next.

Consider $M^7$ a compact oriented smooth manifold. Then, the equations (\ref{eq:Killing spinors 10dim}), (\ref{eq:Anomaly}) and (\ref{eq:instanton}) 
are equivalent to the following system \cite{FrIv03}:
\begin{equation}  \label{eq:systemG2Killing}
  \begin{split}
    d\om\wedge\om & =0, \ \ \ \ \ \ \ \ \ \ \ \ \ d * \om = -4 d \phi \wedge *\om ,\\
    F_A \wedge *\om & = 0, \ \ \ \ \ \ \ \ \ \ \ \  R_\nabla \wedge *\om  = 0,\\
      dH  & = \tr(F_A\wedge F_A)-\tr(R_\nabla\wedge R_\nabla),
  \end{split}
\end{equation}
where $\omega$ is a postive $3$-form that defines a $G_2$ structure on $M$, $-4d\phi$ is the Lee form $\theta_\omega$ of $\omega$, and $H$ is the torsion of the $G_2$-structure, given by
$$
H = - *( d\om - \theta_\omega \wedge \om ).
$$
The first line of equations in \eqref{eq:systemG2Killing} characterizes a special type of $G_2$-structures, namely cocalibrated $G_2$ structures of type $W3$, according to the classification by Fernandez and Gray \cite{FerGr}.
Some Riemannian properties of these structures are studied in \cite{FrIv03}.
The second line of equations in \eqref{eq:systemG2Killing} is the $G_2$-instanton condition, and has been the subject of important recent progress (see e.g. \cite{DS,HN,O,SW},
and the references therein). The last line, the Bianchi identity, is a defining equation for a Courant algebroid, 
and leads to a new mathematical approach to equations from string theories and supergravity theories using methods from generalized geometry (see e.g. \cite{GF,grt,G3}).

Basic compact solutions to the $7$-dimensional Strominger system \eqref{eq:systemG2Killing} are provided by torsion-free $G_2$-structures. For this, one sets $K = G_2$ and $P_K$ the bundle of orthogonal frames of a $G_2$-holonomy metric, and defines $\nabla = A$ equal to the Levi-Civita  connection. An interesting open question is whether one can deform a torsion-free solution, whereby defining a torsion solution of the $7$-dimensional Strominger system (along the lines of the main result in \cite{LiYau}). The first compact solutions with non-zero torsion (and constant dilaton function $\phi$) to the $7$-dimensional Strominger system \eqref{eq:systemG2Killing} have been constructed in \cite{FIUV7}. Non-compact solutions to \eqref{eq:systemG2Killing} have been constructed in \cite{FIUVa7,GuNi}.

We next state our main result, concerning the $7$-dimensional Strominger system \eqref{eq:systemG2Killing}. Let $P_M$ be the bundle of oriented frames over $M$.
The group $\ctG:=\Aut(P_M\times_M P_K)$ acts naturally on the set of parameters $(\omega,\phi,\nabla,A)$ for the system (\ref{eq:systemG2Killing}), preserving solutions, and thus defining a natural set
$$
\mathcal{M} = \{(\omega,\phi,\nabla,A)\;  \mathrm{satisfying} \eqref{eq:systemG2Killing}\}/\ctG.
$$
In this paper, we give the first steps towards the construction of a natural structure of smooth manifold on $\mathcal{M}$. For this, using elliptic operator theory we construct an elliptic complex of differential operators whose cohomology provides an ambient space for the (expected) tangent of $\mathcal{M}$ at $(\omega,\phi,\nabla,A)$. More precisely, we construct a finite-dimensional space of infinitesimal deformations of a solution $(\omega,\phi,\nabla,A)$ of \eqref{eq:systemG2Killing}, modulo the action of $\ctG$. 
\begin{theorem}
 Let $M$ be a $7$-dimensional compact manifold. Then the moduli space of solutions to the system of equations (\ref{eq:systemG2Killing}) on $M$ modulo the 
 $\ctG$-action is finite-dimensional.
\end{theorem}

The Bianchi identity motivates the introduction of generalized geometry to study solutions to (\ref{eq:systemG2Killing}). An equivalent formulation for this system
is provided in Section \ref{sec:genGeom} by mean of Killing spinors on a fixed Courant algebroid. 
As a corollary, the system (\ref{eq:systemG2Killing}) turns out to be a natural system of equations in generalized geometry.
In Section \ref{sec:genGeom}, we provide a relation between the different point of views at the level of moduli spaces.

Further motivation for the study of the moduli space of the system \eqref{eq:systemG2Killing} comes from physics. In this context, a solution of \eqref{eq:systemG2Killing} describes a half-BPS domain wall solution of heterotic supergravity with flux in four dimensions, that preserve $N=1/2$ supersymmetry \cite{OssaLaSv,GLL}. This type of models is particularly appealing in heterotic string theory, since the resulting vacuum breaks supersymmetry in a controlled way, and the low-energy dynamics still allow some of the methods of an $N=1$ four-dimensional effective field theory.


\textbf{Acknowledgments:} CT is partially supported by ANR project EMARKS No ANR-14-CE25-0010 and by CNRS grant PEPS jeune chercheur D-2016-45. AC would like to acknowledge the hospitality of the LMBA of the Universit\'e de Bretagne Occidentale during his visit. MGF is supported by a Marie Sklodowska-Curie grant (MSCA-IF-2014-EF-655162), from the European Union's Horizon 2020 research and innovation programme.


\section{Infinitesimal Moduli of the Strominger system in dimension $7$}
\label{sec:moduli}
Let $M$ be a $7$ dimensional oriented compact manifold. Let $P$ be a principal bundle over $M$ with structure group $G$. Fix a non-degenerate biinvariant pairing 
$$
c \colon \mathfrak{g} \otimes \mathfrak{g} \to \mathbb{R}
$$
on the Lie algebra of $G$. In this section we study the moduli space of solutions to the \emph{Strominger system in dimension $7$}:
\begin{equation}\label{eq:G2 system}
  \begin{split}
    d\om\wedge\om & =0, \ \ \ \ \ \ \ \ \ \ \ \ \ \ d * \om = -4 d\phi \wedge *\om ,\\
    -d(*( d\om + 4 d\phi \wedge \om ))  & = c(F_\theta\wedge F_\theta),\\
     F_\theta \wedge *\om & = 0,
  \end{split}
\end{equation}
where $\om\in \Om^3$ defines a $G_2$ structure, $\phi\in \cC^\infty(M)$, $\theta$ is a connection in $P$, and $F_\theta$ denotes the curvature of $\theta$. 

The first line of equations in \eqref{eq:G2 system} implies that $-4d\phi$ is necessarily the Lee form of $\om$ (see \cite[Proposition 1]{Bryant}), and 
$$
H=-*( d\om + 4 d\phi \wedge \om )
$$ 
is the associated torsion $3$-form. Note that the last equation in \eqref{eq:systemG2Killing}, known as the Bianchi identity, imposes the vanishing of the first Pontryagin class of $P$, calculated using the  biinvariant pairing $c$ on $\mathfrak{g}$
$$
p_1(P)= 0.
$$

As mentioned in Section \ref{sec:intro}, in physics the principal bundle $P$ is taken to be a product $P_M\times_M P_K$, of the bundle of frames $P_M$ of $M$ by a principal bundle $P_K$ over $M$, with compact structure group $K$. Furthermore, $c$ is taken to be of  the form
\begin{equation}\label{eq:pairingc}
  c = 2\alpha'(- \tr_\mathfrak{k} -  c_{\mathfrak{gl}}),
\end{equation}
where $\alpha'$ is a positive constant, $- \tr_\mathfrak{k}$ denotes the Killing form on $\mathfrak{k}$
and $c_{\mathfrak{gl}}$ is a non-degenerate invariant metric on
$\mathfrak{gl}(7,\RR)$, which extends the non-degenerate Killing form
$-\tr$ on $\mathfrak{sl}(7,\RR) \subset \mathfrak{gl}(7,\RR)$. Hence, in this case the aforementioned topological  constraint for the Bianchi identity is equivalent to
$$
p_1(P_M) = p_1(P_K).
$$

\begin{remark}
An additional condition which appears in the physics literature is that $\theta$ is a product connection $\theta = \nabla \times A$, with induced linear connection $\nabla$ on $M$ satisfying $\nabla g = 0$. This is an additional source of complications for the analysis that we ignore in the present paper, and which can be treated with the methods introduced in \cite{grt}.
\end{remark}


\subsection{Main Result}

The parameter space $\cP$ for the equations (\ref{eq:G2 system}) is:
\begin{equation*}
 \label{eq:parameters}
 \cP= \Om^3_{>0}\times \cC^\infty(M) \times \cA
\end{equation*}
where $\Om^3_{>0}$ is the space of positive $3$-forms on $M$ \cite{Hit0} and $\cA$ is the space of 
connections $\theta$ on $P$.

Let $\Diff$ be the group of diffeomorphisms of $M$ and $\cG$ the gauge group of $P$. 
The group of symmetries that we consider is the extension
\begin{equation*}
 \label{eq:symgroup}
 1 \rightarrow \cG \rightarrow \ctG \rightarrow \Diff \rightarrow 1
\end{equation*}
or equivalently, $\ctG=\Aut(P)$. Note that $\ctG$ acts on $\cP$ preserving solutions of (\ref{eq:G2 system}).
Set
\begin{equation}
 \label{eq: equations}
 \begin{array}{cccc}
   \cE : & \cP & \rightarrow & \Om^7 \times \Om^5 \times \Om^4 \times \Om^6(\ad P) \\
         & (\om,\phi,\theta) & \mapsto & (\cE_1,\cE_2, \cE_3,\cE_4)
 \end{array}
\end{equation}
where $\cE_i$ is the ith equation of (\ref{eq:G2 system}). Then
\begin{equation*}
 \label{eq:defmoduli}
 \cM:= \cE^{-1}(0)/ \ctG
\end{equation*}
is the moduli space of solutions modulo symmetries. 
Let $x=(\om, \phi, \theta)\in \cP$ be a solution of $\cE=0$.
Denote by $\bfL$ the linearization of the equations $\cE$ at $x$, and by $\bfP: \Lie \ctG \rightarrow T_x\cP$ the infinitesimal action
of $\ctG$ on $\cP$ at $x$.
As a first approximation to $\cM$, we set
$$
H^1(KS_7):=\frac{\ker \bfL}{\Im \bfP}.
$$

\begin{theorem}
 \label{theo: finite dim}
 The space $H^1(KS_7)$ is finite dimensional.
\end{theorem}

\begin{remark}
Using Theorem \ref{theo: finite dim}, and following Kuranishi's work \cite{ku}, 
it is possible to build a local slice to the $\ctG$-orbits in $\cP$ through a point $x\in\cP$ solution to
(\ref{eq:G2 system}). Then, the local moduli space of solutions around $x$ will be in correspondence with an
analytical subset of the slice, quotiented by the action of the isotropy group of $x$.
\end{remark}


\subsection{Strategy of the proof}
The proof of Theorem \ref{theo: finite dim} relies on elliptic operator theory.
First, we note that at the level of symbols, the equations (\ref{eq:G2 system}) form a system of uncoupled equations for 
the parameters $(\om,\phi)$ on one hand and the parameter $\theta$ on the other hand.
This fact is also true for the infinitesimal action of $\ctG$: at the level of symbols, this action is the product action of
$\Diff$ on one hand and $\cG$ on the other hand. Thus, the proof of Theorem \ref{theo: finite dim} reduces to the separate study of infinitesimal variations of cocalibrated $G_2$ structures of type $W3$ modulo the $\Diff$-action and 
of infinitesimal variations of the $G_2$ instanton equation
for a fixed cocalibrated $G_2$ structures of type $W3$.


\section{$G_2$-holonomy metrics revisited}

Let $M$ be a $7$ dimensional oriented compact manifold. In this section we study cocalibrated $G_2$ structures of type $W3$, in the special case that torsion $3$-form $H$ is closed. In other words, we consider the Strominger system \eqref{eq:G2 system} in the case of trivial structure group $G = \{1\}$
\begin{equation}
 \label{eq:systemnobundle}
  \begin{split}
    d\om\wedge\om & =0, \ \ \ \ \ \ \ \ \ \ \ \ \ \ d * \om  +4 d\phi \wedge *\om=0 ,\\
    -d(*( d\om + 4 d\phi \wedge \om ))  & = 0.
  \end{split}
\end{equation}
This case study will be used in Section \ref{sec: proof main theo} to prove the finiteness of the infinitesimal moduli for the Strominger system in seven dimensions \eqref{eq:systemG2Killing}.

Despite the complicated conditions in \eqref{eq:systemnobundle}, its solutions correspond essentially to $G_2$-holonomy metrics. 

\begin{prop}\label{prop:zeroflux}
A pair $(\omega,\phi)$ is a solution of \eqref{eq:systemnobundle} on a compact $7$-manifold $M$ if and only if $\phi$ is constant and $\omega$ is torsion-free, that is, $d \omega = 0$ and $d^*\omega = 0$.
\end{prop}

This fact is well-known in the physics literature (see e.g. \cite{GaMaWa}). We give a short proof based on two methods for calculating the scalar curvature of a solution of the system \eqref{eq:systemnobundle}, one coming from the relation between Killing spinors in $7$ dimensions and conformally coclosed $G_2$-structures, and the other specifically considering the equations of motion in heterotic string theory implied by \eqref{eq:systemnobundle} (see \cite{Ivan09}).




\begin{proof}[Proof of Proposition \ref{prop:zeroflux}]

From Theorem 1.1 of \cite{Ivan09} a solution of the system \eqref{eq:systemnobundle} for $\omega$ a {\sl positive}  $3$-form determines a metric $g$ with Ricci curvature given by, for the above convention on the Lee form $\theta_\om=-4d\phi$,
\begin{eqnarray*}
\operatorname{Ric}^g_{ij} =\frac{1}{4}H_{imn}H^{mn}_j +4 \nabla_i\nabla_j\phi.
\end{eqnarray*}
Taking the trace with respect to the metric $g$ we obtain 
\begin{eqnarray*} 
S^g=\frac{1}{4}|H|^2-4\Delta\phi
\end{eqnarray*}
where $\Delta =d^* d$ is the Laplacian with positive spectrum. A different expression is calculated as Equation (1.5) of \cite{FrIv03}, without the assumption that $dH=0$, 
\begin{eqnarray*}
S^g = 32|d\phi|^2-\frac{1}{12}|H|^2-12\Delta \phi.
\end{eqnarray*}
Combining these equations we obtain
\begin{eqnarray*}
32|d\phi|^2-\frac{1}{3}|H|^2-8\Delta \phi=0.
\end{eqnarray*}
However, for any $\alpha\in\Real$, and for $\Delta=d^* d$,
\begin{eqnarray*}
\Delta (e^{\alpha \phi})=-\alpha^2e^{\alpha\phi}|d\phi|^2+\alpha e^{\alpha\phi}\Delta \phi,
\end{eqnarray*}
and so
\begin{eqnarray*}
-8\Delta(e^{\alpha\phi})+8(4\alpha-\alpha^2)e^{\alpha\phi}|d\phi|^2-\frac{1}{3}\alpha e^{\alpha\phi}|H|^2 =0. 
\end{eqnarray*}
We can take $\alpha=5$ and integrate over $M$ to conclude that $H=0$ and $d\phi=0$, which imply that 
 $d \omega = 0$ and $d^* \omega = 0$ by \eqref{eq:systemnobundle}. 
\end{proof}

Next, we consider the deformation problem for solutions of \eqref{eq:systemnobundle}, and characterize the space of infinitesimal deformations of this system. By Proposition \ref{prop:zeroflux}, we recover with different methods classical results by Joyce and Hitchin about the infinitesimal moduli of $G_2$ holonomy metrics \cite{joy,Hit0}. As mentioned earlier, our calculations here will be used in Section \ref{sec: proof main theo} to prove the finiteness of the infinitesimal moduli for the Strominger system in seven dimensions \eqref{eq:systemG2Killing}.

Let $\Om^3_{>0}$ denote the space of positive $3$-forms on $M$ \cite{Hit0}. 
Consider the following parameter space $\cP_M$ for the deformation problem for the equations (\ref{eq:G2 system})
\begin{equation*}
 \label{eq:parameters no bundle}
 \cP_M= \Om^3_{>0}\times \cC^\infty(M).
\end{equation*}
Let $\Diff$ be the group of diffeomorphisms of $M$. Then $\Diff$ acts on $\cP_M$ on the left by push-forward preserving solutions of \eqref{eq:systemnobundle}.

Let $(\omega,\phi)$ be a solution of \eqref{eq:systemnobundle}. 
Let $\bfL_M$ be the linearisation of the operator  corresponding to the left hand side of equations \eqref{eq:systemnobundle} 
at the point $(\om,\phi)$, and let $\bfP_M$ be the infinitesimal action of $\Diff$ at this point. 
Using Proposition \ref{prop:zeroflux}, we have explicit formulae:
\begin{equation}
 \label{eq:inf action no bundle}
 \begin{array}{cccc}
 \bfP_M : & \Om^0(T) & \rightarrow & \Om^3\times \cC^\infty(M) \\
        &  V & \mapsto & (d\iota_V\om, 0) 
 \end{array}
\end{equation}
and
\begin{equation}
 \label{eq: linearisation no bundle}
 \begin{array}{cccc}
 \bfL_M : &   \Om^3\times \cC^\infty(M) & \rightarrow &\Om^7 \times \Om^5 \times \Om^4  \\
        &  (\dot \om,\dot \phi) & \mapsto & 
        \begin{cases}
        d\dot \om \wedge \om \\
        d * J \dot  \om+4 d\dot \phi \wedge *\om  \\
       -d(*( d\dot\om + 4d\dot\phi \wedge \om ))
\end{cases}
\end{array}
\end{equation}
where for $l=3,4$, $J:\Omega^l\rightarrow \Omega^l$
is defined by
\begin{equation}
\label{eq:definitionJ}
J(\xi) = \frac{4}{3}\pi_1(\xi)+\pi_7(\xi)-\pi_{27}(\xi)
\end{equation}
and $\pi_k$ denotes the projection onto the $k$ dimensional component of $\Om^l$ (see e.g. \cite{joy}
for the variation of $*\om$). 

Note that $\bfL_M$ is a multi-degree differential operator \cite{DN}. Given $v\in T^*\setminus M$, we have the following:
\begin{lemma}
\label{lem:symbols P}
 The tuples $\mathbf{t}_P=(1,1)$ and $\mathbf{s}_P=(0,0)$ form a system of orders for $\bfP_M$.
 The $(\mathbf{t}_P,\mathbf{s}_P)$ principal symbol of the linearisation of the infinitesimal action of $\Diff$ is
   \begin{equation}
   \label{eq:symbolP}
    \sigma_{\mathbf{P}_M}(v)(V)= (v\wedge \iota_V\om, 0)
  \end{equation}
\end{lemma}

\begin{lemma}
\label{lem:symbols L}
 The tuples $\mathbf{t}_L=(2,2)$ and $\mathbf{s}_L=(1,1,0)$ form a system of orders for $\bfL_M$.
 The $(\mathbf{t}_L,\mathbf{s}_L)$ principal symbol of $\mathbf{L}_M$ is
  \begin{equation} 
  \label{eq:symbolL}
    \sigma_{\mathbf{L}_M}(v)(\dot\om,\dot \phi)= (v\wedge \dot\om\wedge \om, v\wedge(*J\dot \om + \dot \phi *\om),v\wedge *(v\wedge(\dot\om+\dot\phi\om)))
  \end{equation}
\end{lemma}

Set $\cR_M=\Om^7 \times \Om^5 \times \Om^4$.
We can now prove the main result of this section.
\begin{prop}
 \label{prop: W3 G2 metric moduli}
The sequence:
\begin{equation}\label{eq:complexnoflux}
 \Lie(\Diff) \stackrel{\bfP_M}{\longrightarrow} T_{(\om,0)}\cP_M \stackrel{\bfL_M}{\longrightarrow} \cR_M
\end{equation}
is elliptic in the middle. Furthemore,
\begin{equation}
 \label{eq:tangent H3 R}
 \frac{\ker \bfL_M}{\Im \bfP_M}\simeq \cH^3(M,\RR)\times \RR
\end{equation}
 where $\cH^i(M,\RR)$ is the space of harmonic $i$-forms on $(M,g)$.
 \end{prop}
 
\begin{proof}
 We start proving the isomorphism \eqref{eq:tangent H3 R}. Let $(\dot \om, \dot \phi)\in T_{(\om,0)}\cP_M$ be in the kernel of $\bfL_M$.
 From 
 $$d(*( d\dot\om + 4d\dot\phi \wedge \om ))=0\; \text{ and }\; d\om=0$$
 we deduce 
 $$d^*d(\dot \om + 4 \dot \phi \om)=0,$$
 and thus
 $\dot \om + 4 \dot \phi \om$ is closed. By Hodge decomposition, there exists a two form $\beta\in \Om^2$ and a
 harmonic $3$-form $h$ such that
 \begin{equation}
  \dot \om + 4 \dot \phi \om = h + d\beta.
 \end{equation}
By Proposition \ref{prop:zeroflux}, $\om$ is a $G_2$ metric, and therefore the decomposition of forms given by the $G_2$ representation commutes with the Hodge Laplacian.
Thus $*$ and $J$ preserve harmonic forms. Recall that (see e.g. \cite{Bryant})
$$
\Om^2=\Om^2_7\oplus\Om^2_{14}
$$
and decompose $\beta$ into
$7$th and $14$th components, $\beta=\beta_7+\beta_{14}$.
Equation
$$
 d * J \dot  \om+4 d\dot \phi \wedge *\om =0
 $$
 then implies
\begin{equation}
 \label{eq: d*Jdbeta}
 d*Jd\beta_7+ d*Jd\beta_{14}-\frac{4}{3} d\dot\phi \wedge *\om=0.
\end{equation}
On the other hand, using $\beta_7=*(\alpha\wedge *\om)$ for some one form $\alpha$, and Bryant's work \cite[Prop. 3]{Bryant},
 we obtain 
 $$d*Jd\beta_7=0\: \text{ and }\: \pi_7(d*Jd\beta_{14})=0.$$
 Thus (\ref{eq: d*Jdbeta}) is equivalent to
 $$
 \pi_{14}(d*Jd\beta_{14})-\frac{4}{3} d\dot\phi \wedge *\om=0.
 $$
Using type decomposition on $\Om^5$, and \cite[Prop. 3]{Bryant}, we deduce:
\begin{equation}
 \label{eq: laplacian beta 14}
 \pi_{14}(d*Jd\beta_{14})=\Delta \beta_{14} -\frac{3}{2}d_{14}^7d_7^{14}\beta_{14}=0
\end{equation}
and
$$d\dot \phi =0.$$
Thus $\dot \phi$ is constant. Moreover, as $\beta_{14}\in \Om^2_{14}$ is equivalent to $\beta_{14}\wedge \om=- *\beta_{14}$, we have
$$
d_{14}^7d_7^{14}\beta_{14}=\pi_{14}(d*d(\beta_{14}\wedge\om))=-\pi_{14}(dd^*\beta_{14}).
$$
Using the fact that
$$
\pi_{14}=\frac{2}{3} Id - \frac{1}{3} *(\cdot \wedge \om)
$$
equation (\ref{eq: laplacian beta 14}) becomes:
$$
\Delta \beta_{14} + dd^* \beta_{14} + \frac{1}{2}d^*(*(d^*\beta_{14}\wedge\om))=0.
$$
Using Hodge decomposition into orthogonal components, we deduce that $d^*\beta_{14}=0$
and $d\beta_{14}=0$.
Returning to $\dot\om$, and by definition of $\Om^2_7$, there exists $V\in\Om^0(T)$ such that
$$
 \dot \om + 4 \dot \phi \om = h+d\beta_7=h + d\iota_V\om.
$$
Thus we can define a map 
\begin{equation}
\begin{array}{ccc}
\ker \bfL_M  & \rightarrow & \cH^3\times \RR\\
(\dot\om,\dot \phi) & \mapsto & (h-4\dot \phi \om, \dot \phi).
 \end{array}
\end{equation}
This map is well defined, surjective, and has kernel the image of $\bfP_M$, thus proving \eqref{eq:tangent H3 R}.

Using now the same argument at the level of symbols, combined with Lemma \ref{lem:symbols P} and Lemma \ref{lem:symbols L}, gives a proof of the ellipticity of \eqref{eq:complexnoflux}.
More details are given in the proof of Proposition \ref{prop: elliptic no bundle}.
\end{proof}

\section{Proof of Theorem \ref{theo: finite dim}}
\label{sec: proof main theo}
Let $x=(\om,\phi,\theta)\in \cP$. Then the tangent of $\cP$ at $x$ is:
$$
T_x\cP= \Om^3\times \cC^\infty(M) \times \Om^1(\ad P).
$$
Set $\cR:=\Om^7 \times \Om^5 \times \Om^4 \times \Om^6(\ad P)$.
To prove Theorem \ref{theo: finite dim}, it is enough to show that the sequence
\begin{equation}
 \label{seq: P and L}
 \Lie (\ctG)  \stackrel{\bfP}{\longrightarrow} T_x\cP \stackrel{\bfL}{\longrightarrow} \cR
\end{equation}
is elliptic 
at a point $x \in \cP$ which is a solution to $\cE(x)=0$.

\subsection{Linearisation and symbols}
\label{sec: lin and symbols}
The infinitesimal action $\bfP$ and the linearisation $\bfL$ of $\cE$ are given by:
\begin{equation}
 \label{eq:inf action}
 \begin{array}{cccc}
 \bfP : & \Om^0(T)\times \Om^0(\ad P) & \rightarrow & \Om^3\times \cC^\infty(M) \times \Om^1(\ad P)\\
        &  (V, r) & \mapsto & (L_V\om, L_V \phi, d^\theta r + \iota_V F_\theta) 
 \end{array}
\end{equation}
and 
\begin{equation}
 \label{eq: full linearisation}
 \begin{array}{cccc}
 \bfL : &   \Om^3\times \cC^\infty(M) \times \Om^1(\ad P) & \rightarrow &\Om^7 \times \Om^5 \times \Om^4 \times \Om^6(\ad P)  \\
        &  (\dot \om,\dot \phi, \dot \theta) & \mapsto & 
        \begin{cases}
        \bfL_1= d\dot \om \wedge \om + d\om\wedge \dot \om \\
        \bfL_2= d * J \dot  \om+4 d\dot \phi \wedge *\om +4 d\phi \wedge *J \dot \om \\
        \bfL_3= -d(*( d\dot\om + 4d\dot\phi \wedge \om )) -d(\dot *( d\om + 4 d\phi \wedge \om ))\\
        \:\: -d(*( 4 d\phi \wedge \dot \om )) -2d(c(\dot\theta, F_\theta))\\
        \bfL_4= d^\theta\dot \theta \wedge *\om + F_\theta \wedge *J \dot \om
\end{cases}
 \end{array}
\end{equation}
where $d^\theta$ denotes the extension of the covariant derivative associated to $\theta$ on $\Om^*(\ad P)$;
and recall $J$ is defined by (\ref{eq:definitionJ}).
We will use the theory of linear multi-degree elliptic
differential operators \cite{LM1,LM2}. At the level of symbols, only the highest order operators will appear. Thus, the symbols of $\bfL$ and $\bfP$ are
the same as the symbols of $\bfL_h$ and $\bfP_h$ defined by:
\begin{equation}
 \label{eq:inf action high}
 \begin{array}{cccc}
 \bfP_h : & \Om^0(T)\times \Om^0(\ad P) & \rightarrow & \Om^3\times \cC^\infty(M) \times \Om^1(\ad P)\\
        &  (V, r) & \mapsto & (d\iota_V\om, \iota_Vd \phi, d^\theta r) 
 \end{array}
\end{equation}
and 
\begin{equation}
 \label{eq: full linearisation high}
 \begin{array}{cccc}
 \bfL_h : &   \Om^3\times \cC^\infty(M) \times \Om^1(\ad P) & \rightarrow &\Om^7 \times \Om^5 \times \Om^4 \times \Om^6(\ad P)  \\
        &  (\dot \om,\dot \phi, \dot \theta) & \mapsto & 
        \begin{cases}
        d\dot \om \wedge \om \\
        d * J \dot  \om+4 d\dot \phi \wedge *\om  \\
       -d(*( d\dot\om + 4d\dot\phi \wedge \om )) \\
        d^\theta\dot \theta \wedge *\om
\end{cases}
 \end{array}
\end{equation}
Note that for the later equations, the parameters $(\om,\phi)$ and $\theta$ are uncoupled. We next construct two elliptic complexes 
\begin{equation}
\label{seq:1}
\Lie(\Diff)\stackrel{\bfP_M}{\longrightarrow} T_{(\om,\phi)}\cP_M \stackrel{\bfL_M}{\longrightarrow} \cR_M
\end{equation}
 \begin{equation}
 \label{seq:2}
 \Lie(\cG) \stackrel{{\bf P}_P}{\longrightarrow} T_{\theta}\cA \stackrel{\bf L_P}{\longrightarrow} \cR_P
\end{equation}
so that $\bfL_h = \bfL_M \times \bfL_P$ and $\bfP_h = \bfP_M \times \bfP_P$ 
(and hence at the level of symbols $\sigma_{\bfL_h}=\sigma_{\bfL_M}\times\sigma_{\bfL_P}$ and 
$\sigma_{\bfP_h}=\sigma_{\bfP_M}\times\sigma_{\bfP_P}$). 
From our previous considerations, 
Theorem \ref{theo: finite dim} follows from the ellipticity of \eqref{seq:1} and \eqref{seq:2}, that we prove respectively in Proposition \ref{prop: elliptic no bundle} and Proposition \ref{prop: elliptic only bundle}.

\subsection{Ellipticity of (\ref{seq:1})}
The first complex \eqref{seq:1} is closely related to the linearisation of the system of equations (\ref{eq:systemnobundle}).
However we do not assume that $(\om,\phi)$ is a solution to (\ref{eq:systemnobundle}).
We set $\bfP_M$ and $\bfL_M$ to be the operators defined in \eqref{eq:inf action no bundle} and \eqref{eq: linearisation no bundle},
with a slight modification: $\bfP_M(V)=(d\iota_V\om, \iota_Vd\phi)$, as $\phi$ might not be constant.
\begin{prop}
 \label{prop: elliptic no bundle}
 Let $(\om,\phi)\in \cP_M$. Then the sequence
 \begin{equation}
 \label{seq: P_M and L_M}
 \Lie(\Diff)\stackrel{\bfP_M}{\longrightarrow} T_{(\om,\phi)}\cP_M \stackrel{\bfL_M}{\longrightarrow} \cR_M
\end{equation}
is elliptic in the middle.
\end{prop}

\begin{proof}
Let $v\in T^*\setminus M$.
 Note that for the same choice of tuples as in Lemmas \ref{lem:symbols P} and \ref{lem:symbols L},
 the principal symbols for $\bfP_M$ and $\bfL_M$ are the one given in equations (\ref{eq:symbolP})
 and (\ref{eq:symbolL}). Assume that
 $$
 \sigma_{\mathbf{L}_M}(v)(\dot\om,\dot \phi)= (0,0,0).
 $$
 We want to show that $\dot \phi=0$ and $\dot\om=v\wedge \iota_V\om$ for some $V\in T$, or equivalently
 $\dot\om=v\wedge \beta$ for some $\beta\in \Lambda^2_{7}$. The proof follows
 the one of Proposition \ref{prop: W3 G2 metric moduli}, at the symbol level.
 From
 $$
 v\wedge *(v\wedge(\dot\om+\dot\phi\om)))=0
$$
 we deduce
 $$
 0=\langle v\wedge *(v\wedge(\dot\om+\dot\phi\om)), *(\dot\om+\dot\phi\om)\rangle =\langle v\wedge(\dot\om+\dot\phi\om),v\wedge(\dot\om+\dot\phi\om)\rangle
 $$
 and thus
 $$
 v\wedge(\dot\om+\dot\phi\om)=0.
 $$
 There exists $\beta_7\in \Lambda^2_7$ and $\beta_{14}\in\Lambda^2_{14}$ such that
 $$
 \dot\om+\dot\phi\om=v\wedge \beta_7 + v\wedge \beta_{14}.
 $$
 Together with the equation
 $$
 v\wedge(*J\dot \om + \dot \phi *\om)=0
 $$
 we obtain
 \begin{equation}
  \label{eq:someequation involving symbols and betas}
  v\wedge *(J(v\wedge\beta_7)) + v\wedge *(J(v\wedge \beta_{14}) )=\frac{4}{3}\dot\phi v\wedge * \om.
 \end{equation}
Now, note that if we don't assume the $3$-form $\om$ to be torsion free,
the formulas in \cite[Proposition 3]{Bryant} are only modified by lower order terms.
As noticed in the proof of Proposition \ref{prop: W3 G2 metric moduli}, these formulas imply that
for $2$-forms $\beta_7\in\Om^2_7$ and $\beta_{14}\in\Om^2_{14}$, if $\om$ is torsion-free, then
$$d*Jd\beta_7=0\: \text{ and }\: \pi_7(d*Jd\beta_{14})=0.$$
In general, these operators might not vanish. However, they are operators of degree less or equal to one.
This implies that the quadratic part of the symbol of $d*Jd$ restricted to $\Om^2_7$ vanishes,
as well as the quadratic part of the symbol of $\pi_7(d*Jd)$ restricted to $\Om^2_{14}$,
even if $\om$ has torsion.
Thus,
$$v\wedge*J(v\wedge\beta_7)=0\: \text{ and }\: \pi_7(v\wedge*J(v\wedge\beta_{14}))=0.$$
Formula (\ref{eq:someequation involving symbols and betas}) becomes
$$
\pi_{14}(v\wedge *(J(v\wedge \beta_{14})))=\frac{4}{3}\dot\phi v\wedge * \om.
$$
As $v\wedge * \om$ is in $\Lambda^2_7$, we have $\dot\phi=0$ and thus $v\wedge *(J(v\wedge \beta_{14}))=0$.
Then, as in the proof of Proposition \ref{prop: W3 G2 metric moduli}, we obtain
$$
\sigma_\Delta(v) \beta_{14} + \sigma_{dd^*}(v) \beta_{14} + \frac{1}{2}\sigma_{l}(v)\beta_{14}=0,
$$
where $l=d^*(*(d^*\cdot\wedge\om))$. We conclude similarly that $v\wedge \beta_{14}=0$, which ends the proof.
\end{proof}

\subsection{$G_2$-instantons}
\label{sec:G2 instantons moduli}

The second complex \eqref{seq:2} corresponds to a system parameterizing $G_2$ instantons on $P$ modulo the gauge group $\cG$, for a cocalibrated $G_2$ structures of type $W3$. For the proof, we can indeed assume that $\om$ is an arbitrary $G_2$ structure. The equation $\cE_P$ we consider on $\cA$ is:
\begin{equation}
 \label{eq: system only bundle}
F_\theta \wedge *\om=0
\end{equation}
Set as before $\bfL_P$ the linearisation of $\cE_P$ at $\theta\in \cA$, $\bfP_P$ the infinitesimal action
of $\cG$ on $\cA$ at $\theta$ and $$\cR_P=\Om^6(\ad P).$$
\begin{prop}
\label{prop: elliptic only bundle}
 Assume $(\theta)$ is a solution to (\ref{eq: system only bundle}). Then the sequence
 \begin{equation}
 \label{seq: P_P and L_P}
 \Lie(\cG) \stackrel{{\bf P}_P}{\longrightarrow} T_{\theta}\cA \stackrel{\bf L_P}{\longrightarrow} \cR_P
\end{equation}
is elliptic at the middle term.
\end{prop}
This is to say that the image of 
the symbol $\sigma({\bf P}_P)$, which equals $(\ker \sigma({\bf P}_P^*))^\perp$, is equal to the kernel of $\sigma({\bf L}_P)$. This is equivalent to the fact that the symbol of the differential operator $D_\theta={\bf L_P}+{\bf P}_P^*$ is injective. 

\begin{lemma}
\label{lem:ellipticbundle}
The symbol of $D_\theta={\bf L_P}+{\bf P}_P^*$ given by  
\begin{eqnarray*}
\sigma_{D_\theta}(\xi)a=(*\omega\wedge(\xi\wedge a),\iota_{\xi^\#}a)
\end{eqnarray*}
is injective.
\end{lemma}


\begin{proof}
$D_\theta$ is given by
\begin{eqnarray*}
D_\theta a=(*\omega\wedge d^\theta a ,(d^\theta)^*a)
\end{eqnarray*}
so the expression for the symbol is clear. 
To see that the map is injective we neglect the coefficient bundle $\ad P$ and observe that if $\iota_{\xi^\#}a=0$ then $\xi,a\in\Lambda^1$ are orthogonal. $G_2$ acts transitively on the set of ordered pairs of orthogonal vectors in $\mathbb{R}^7$ so we suppose that $\xi=e^1$ and $a=\lambda e^2$ while $\omega=e^{123}+e^1\wedge(e^{45}-e^{67})+e^2\wedge(e^{46}-e^{75})+e^3\wedge(e^{47}-e^{56})$. Then $*\omega\wedge\xi\wedge a=0$ if and only if $\xi\wedge a\in\Lambda^2_{14}$. The $\Lambda^2_7$ component of $e^1\wedge \lambda e^2$ is given by $\pi_7(e^1\wedge\lambda e^2)=\lambda/3(e^1\wedge e^2-e^4\wedge e^7+e^5\wedge e^6)$ which is non-zero whenever $\lambda\neq 0$.
\end{proof}

\begin{remark}
The usual method for studying the moduli problem for $G_2$-instantons is to consider the $G_2$-monopole equation $F_\theta\wedge *\omega+*d^\theta\Phi=0$ for $\theta$ a connection and $\Phi\in\Omega^0(\ad P)$ a Higgs-type field. 
The monopole equation arises as the dimensional reduction of the $Spin_7$-instanton equation and so defines an elliptic system. 
As observed by Walpuski \cite{Wa}, the Bianchi identity for $\theta$ implies that on a closed manifold $(M,\omega)$ with $G_2$-structure satisfying $d*\omega=0$ a $G_2$-monopole must satisfy $F_\theta\wedge *\omega=0$. 

Under the condition that $d*\omega=-4d\phi\wedge *\omega$, we can define a new positive $3$-form $\tilde{\omega}=e^{3\phi}\omega$. Then $\omega$-instantons coincide exactly with $\tilde{\omega}$-instantons. The form $\tilde{\omega}$ also satisfies $d\tilde{*}\tilde{\omega}=0$ so we can conclude that $\omega$-instantons coincide with $\tilde{*}\tilde{\omega}$-monopoles. The latter, modulo gauge, is a well-determined elliptic condition. We avoid this framework however, to avoid introducing the extraneous Higgs term.
\end{remark}


\section{Killing Spinors and generalized geometry in dimension 7}
\label{sec:genGeom}
Killing spinors in generalized  geometry were
introduced in \cite{grt} and studied on $6$ dimensional manifolds. In this section we show that on a $7$ dimensional manifold the Killing spinor equations on a suitable transitive Courant algebroid are equivalent to the system (\ref{eq:G2 system}).
Furthermore, following \cite{grt} we introduce the space of infinitesimal variations of a solution of the Killing spinor equations
modulo inner symmetries of the Courant algebroid. We relate this space to $H^1(KS_7)$.


\subsection{Courant algebroids, generalized metrics and Killing spinors}
 \label{sec:Courant alg and metrics}
 \begin{defi}
 \label{def:Courant}
  A Courant algebroid $(E,\la\cdot,\cdot\ra,[\cdot,\cdot],\pi)$ over a
  manifold $M$ consists of a vector bundle $E\to M$ together with a
  non-degenerate symmetric bilinear form $\la\cdot,\cdot\ra$ on $E$, a
  (Dorfman) bracket $[\cdot,\cdot]$ on the sections $\Omega^0(E)$, and
  a bundle map $\pi:E\to TM$ such that the following properties are
  satisfied, for $e,e',e''\in \Omega^0(E)$ and $\phi\in \cC^\infty(M)$:
  \begin{itemize}
  \item[(D1):] $[e,[e',e'']] = [[e,e'],e''] + [e',[e,e'']]$,
  \item[(D2):] $\pi([e,e'])=[\pi(e),\pi(e')]$,
  \item[(D3):] $[e,\phi e'] = \pi(e)(\phi) e' + \phi[e,e']$,
  \item[(D4):] $\pi(e)\la e', e'' \ra = \la [e,e'], e'' \ra + \la e',
    [e,e''] \ra$,
  \item[(D5):] $[e,e']+[e',e]=2\pi^* d\la e,e'\ra$.
  \end{itemize}
\end{defi}
We are interested in a particular class of Courant algebroids, that can be constructed as follows. Let $M$ be a smooth manifold of dimension $n$. Let $G$ be a Lie group and $P$ be a principal $G$-bundle over $M$. As in Section \ref{sec:moduli}, we fix a non-degenerate biinvariant pairing $c$ on the Lie algebra $\mathfrak{g}$ of $G$.
Consider the vector bundle
\begin{equation}\label{eq:Edef}
  E = T \oplus \ad P \oplus T^*
\end{equation}
endowed with the symmetric pairing
$$
\langle X + r + \xi,Y + t + \eta\rangle = \frac{1}{2}(\eta(X) +
\xi(Y)) + c(r,t),
$$
and the canonical projection
$$
\pi \colon E \to T.
$$
Given $3$-form $H_0$ on $M$ and a connection $\theta_0$ on $P$ with curvature $F_0$, we can endow $\Omega^0(E)$ with a bracket
\begin{equation}\label{eq:bracket}
  \begin{split}
    [X+r+\xi,Y+t+\eta]  = {} & [X,Y] + L_{X}\eta - i_{Y}d\xi + i_{Y}i_{X}H_0\\
    & - [r,t] - F_0(X,Y) + d^{\theta_0}_Xt - d^{\theta_0}_Y r\\
    & + 2c(d^{\theta_0} r,t) + 2c(F_0(X,\cdot),t) - 2c(F_0(Y,\cdot),r).
  \end{split}
\end{equation} 
Following \cite{ChStXu}, it can be checked that the tuple $(E,\la\cdot,\cdot\ra,[\cdot,\cdot],\pi)$ satisfies the axioms of Definition \ref{def:Courant} if and only if the following Bianchi identity is satisfied
\begin{equation}\label{eq:bianchi}
  d H_0 = c (F_0 \wedge F_0).
\end{equation}

Let $(t,s)$ be the signature of the pairing on the Courant algebroid
$E$. A generalized metric of signature $(p,q)$ is given by a subbundle
$$
V_+ \subset E
$$
such that the restriction of the metric on $E$ to $V_+$ is a
non-degenerate metric of signature $(p,q)$. We denote by $V_-$ the
orthogonal complement of $V_+$ on $E$.
\begin{defi}[\cite{GF}]\label{def:admet}
  A metric $V_+$ of arbitrary signature is admissible if
$$
V_+ \cap T^* = \{0\} \qquad \textrm{and} \qquad \rk V_+ = \rk E - \dim
M.
$$
\end{defi}
A generalized connection $D$ (or simply, a connection) on
$E$ is a first order differential operator
$$
D \colon \Omega^0(E) \to \Omega^0(E^* \otimes E)
$$
satisfying the Leibniz rule $D_e(\phi e') = \phi D_ee' +
\pi(e)(\phi)e'$, for $e,e' \in \Omega^0(E)$ and $\phi\in
C^\infty(M)$ and compatible with the
inner product on $E$, that is, satisfying
$$
\pi(e)(\langle e',e'' \rangle) = \langle D_e e',e'' \rangle + \langle
e',D_e e'' \rangle.
$$
Given an admissible metric and a smooth function $\phi \in C^\infty(M)$, one can associate a torsion-free, compatible connection $D^\phi$ \cite{GF,grt}, constructed from the Gualtieri-Bismut connection \cite{G3}. 
The connection $D^\phi$ induces differential operators
$$
D^\phi_{\pm}: V_-\to V_-\ot (V_\pm)^*.
$$ 
From $D^\phi_+$ and $D^\phi_-$ we get differential
operators on spinors
$$
D^\phi_\pm:S_+(V_-)\to S_+(V_-)\ot (V_\pm)^*
$$
and the associated Dirac operator
$$
\slashed D^\phi_-:S_+(V_-)\to S_-(V_-).
$$
\begin{defi}
  Given a generalized metric $V_+$ and $\phi \in
  C^{\infty}(M)$, the \emph{Killing spinor equations} for a spinor
  $\eta \in S_+(V_-)$ are
  \begin{equation}\label{eq:Killing}
    \begin{split}
      D^\phi_+ \eta &= 0,\\
      \slashed D_-^\phi \eta & = 0.
    \end{split}
  \end{equation}
\end{defi}

Specifying now the previous construction to dimension $n = 7$, we obtain a characterization of the Strominger system \eqref{eq:G2 system}.

\begin{theorem}\label{thm:Strom}
Let $M$ be $7$-dimensional oriented spin manifold, endowed with the Courant algebroid  $(E,\la\cdot,\cdot\ra,[\cdot,\cdot],\pi)$ determined by a pair $(H_0,\theta_0)$ satisfying \eqref{eq:bianchi}.  A solution $(V_+,\phi,\eta)$ of the killing spinor equations \eqref{eq:Killing}
is equivalent to a tuple $(\omega,\phi,\theta)$ satisfying the Strominger system \eqref{eq:G2 system}, where $\omega$ is a $G_2$-structure on $M$, with torsion
\begin{equation}\label{eq:thmStrom2}
      - *( d\om - \theta_\omega \wedge \om) = H_0 + db + 2c (a,F_0) + c(a, d^{\theta_0}a) + \frac{1}{3} c(a,[a , a]),
\end{equation}
for  $a := \theta - \theta_0 \in \Omega^1(\ad P)$ and a suitable $2$-form $b$ on $M$.
\end{theorem}

\begin{proof}
The result follows combining the proof of \cite[Lemma 5.1]{grt} (which can be easily generalized to arbitrary dimension of $M$) with the proof of \cite[Theorem 1.2]{FrIv03}. For the convenience of the reader, we give a sketch of the argument. Following \cite{grt}, an admissible metric is equivalent to a metric $g$ on $M$, together with an isotropic splitting of the anchor map $\pi :E \rightarrow T$, which determines a $2$-form $b \in \Omega^2$ and a $1$-form $a \in \Omega^1(\ad P)$ with values in $\ad P$. Define a $3$-form $H$ and a connection $\theta$ on $P$ by the formulae
\begin{align*}
H & = H_0 + db + 2c (a,F_0) + c(a, d^{\theta_0}a) + \frac{1}{3} c(a,[a , a]),\\
\theta & = \theta_0 + a,
\end{align*}
such that the Bianchi identity $dH=c(F_\theta\wedge F_\theta)$ is satisfied. Then, a solution to the Killing spinor equations in generalized geometry \eqref{eq:Killing} gives a tuple $(g,\theta,H,\eta,\phi)$, where 
$\eta$ is a spinor with respect to $g$ and $\phi$ is the dilaton function. Arguing as in the proof of \cite[Lemma 5.1]{grt}, it follows that tuple satisfies the Killing spinors equations \eqref{eq:Killing spinors 10dim}, jointly with the instanton condition \eqref{eq:instanton} and the Bianchi identity \eqref{eq:bianchi}. Finally, following the proof of \cite[Theorem 1.2]{FrIv03}, $(g,\eta)$ determine a $G_2$ structure $\om$ with torsion $H$, and $(\om, \phi, \theta)$ is a solution to \eqref{eq:G2 system}. Conversely, given a solution $(\om, \phi, \theta)$ of \eqref{eq:G2 system} satisfying \eqref{eq:thmStrom2} for a $2$-form $b$ on $M$, we can associate a generalized metric $V_+$ on the fixed Courant algebroid $E$, determined by $g = g_\omega$, $a = \theta - \theta_0$ and $b$. Then, considering the spinor $\eta$ determined by $\omega$, it follows that $(V_+,\phi,\eta)$ provides a solution of \eqref{eq:Killing}.
\end{proof}

Note that the condition \eqref{eq:thmStrom2} in Theorem \ref{thm:Strom} can be expressed more invariantly by the following equivalent condition in the equivariant cohomology of $P$
$$
[p^*H_0 - CS(\theta_0)] = [p^*H - CS(\theta)] \in H^3(P,\RR)^G,
$$
where $p \colon P \to M$ is the canonical projection and $CS(\theta) \in \Omega^3(\ad P)^G$ is the ($G$-invariant) Chern-Simons $3$-form of the connection $\theta$. The class $[p^*H_0 - CS(\theta_0)] \in H^3(P,\RR)^G$ can be regarded as the isomorphism class of a $G$-equivariant (exact) Courant algebroid on the total space of $P$, from which the (transitive) Courant algebroid $E$ is obtained by reduction \cite{GF}.


\subsection{Infinitesimal moduli for the Killing spinor equations in dimension $7$}
\label{sec: inf considerations}
We now describe the space of infinitesimal solutions to the Killing spinors equations (\ref{eq:Killing})
on a fixed Courant algebroid modulo inner symmetries of the algebroid.

Let $(E,\la\cdot,\cdot\ra,[\cdot,\cdot],\pi)$ be the Courant algebroid determined by a solution $x=(\om,\phi,\theta)$ of $(\ref{eq:G2 system})$, as in Section \ref{sec:Courant alg and metrics} (setting $H_0$ to be the torsion of $\omega$ and $\theta_0 = \theta$).
The group of symmetries $\Aut(E)$ of this transitive Courant algebroid has been described in \cite[Proposition 4.3]{grt}, following \cite{Rubioth}.
In particular, the space $\Om^0(E)$, sections of the bundle $E$, naturally embeds into the Lie algebra
of inner symmetries of the Courant algebroid via the map
$$
\begin{array}{ccc}
 \Om^0(E) & \rightarrow & \Lie(\Aut(E))\\
  e & \mapsto & [\;e\;,\;\cdot\;] .
\end{array}
$$
More explicitly, this is described by
\begin{equation}
 \begin{array}{ccc}
 \Om^0(E)=\Om^0(T)\times \Om^0(\ad P)\times \Om^1 & \rightarrow & \Lie(\Aut(E))\\
  (V,r,\xi) &  \rightarrow & (V,r,-d\xi -\iota_V H + 2c(r,F)).
 \end{array}
\end{equation}
Note also
$$
\Om^0(E)=\Lie(\ctG)\times \Om^1.
$$
An admissible generalized metric is equivalent to a metric $g$ on $M$ together with an isotropic splitting of $E\rightarrow T$.
Such isotropic splitting are locally modelled on $\Om^1(\ad P)\times \Om^2$.
Thus the space of infinitesimal variations of generalized metrics is modeled on $S^2T^*\times\Om^1(\ad P)\times \Om^2$.
We add the dilaton $\phi$ as a parameter defining the connection $D^\phi$ of a generalized metric. We also consider
$G_2$ metrics $g$ defined by a $3$-form $\om$. Then the tangent space to the space of parameters 
for the Killing spinor equations \eqref{eq:Killing} in dimension 7 is:
$$
T_x\hat{\cP}:=\Om^3\times \cC^\infty(M)\times \Om^1(\ad P)\times \Om^2.
$$
Note that
$$
T_x\hat{\cP}=T_x\cP \times \Om^2.
$$
We can consider the infinitesimal action $\hbfP$ of $\Om^0(E)$ on the space of parameters:
\begin{equation}
 \label{eq:gen inf action}
 \begin{array}{cccc}
  \hbfP : & \Om^0(E) & \rightarrow & T_x\hat{\cP} \\
   & (V,r,\xi) & \mapsto & (\bfP(V,r), d\xi + i_V(H) - 2c(r,F)),
 \end{array}
\end{equation}
where we recall the operator $\bfP$ is defined in Section \ref{sec: lin and symbols}.
Note that from the naturality of the equations \eqref{eq:Killing}, solutions
to this equations form orbits under the action of $\Aut(E)$.

We set now $\hat{\bfL}$ to be the linearisation of the Killing spinors equations \eqref{eq:Killing} in dimension $7$ on the
fixed Courant algebroid $(E,\la\cdot,\cdot\ra,[\cdot,\cdot],\pi)$. Fixing the Courant algebroid structure
amounts to allow variations of the $3$-form $H$ by exact terms only, and thus replace the Bianchi identity equation
by a primitive equation:
\begin{equation}
 \label{eq: Gen Kil linearisation}
 \begin{array}{cccc}
 \hat{\bfL} : &   \Om^3\times \cC^\infty(M) \times \Om^1(\ad P)\times \Om^2 & \rightarrow &\Om^7 \times \Om^5 \times \Om^3 \times \Om^6(\ad P)  \\
        &  (\dot \om,\dot \phi, \dot \theta, b) & \mapsto & 
        \begin{cases}
        \hbfL_1= d\dot \om \wedge \om + d\om\wedge \dot \om \\
        \hbfL_2= d * J \dot  \om+4 d\dot \phi \wedge *\om +4 d\phi \wedge *J \dot \om \\
        \hbfL_3= \dot T -2(c(\dot\theta, F_\theta))-db\\
        \hbfL_4= d^\theta\dot \theta \wedge *\om + F_\theta \wedge *J \dot \om
\end{cases}
 \end{array}
\end{equation}
where $\dot T$ stands for the infinitesimal variation of the torsion term $-*( d\om + 4 d\phi \wedge \om )$
with respect to the variation $(\dot\om,\dot\phi)$. 
Then define a sequence of differential operators
\begin{equation}
  \label{eq:Gen complex}
\Om^0(E) \lra{\hbfP} T_x\hat{\cP} \lra{\hbfL} \Om^7 \times \Om^5 \times \Om^3 \times \Om^6(\ad P)
\end{equation}
A straightforward calculation shows that (\ref{eq:Gen complex}) is a complex of differential operators.
Then, relying on Theorem \ref{theo: finite dim}, one easily shows:
\begin{theorem}
 \label{theo: gen finite dim}
 The complex \eqref{eq:Gen complex} is elliptic, and thus the space
 $$
 H^1(\widehat{KS}_7):=\frac{\ker \hbfL}{\Im \hbfP}
 $$
 is finite dimensional.
\end{theorem}

To relate $ H^1(\widehat{KS}_7)$ to the space $H^1(KS_7)$ from Theorem \ref{theo: finite dim}, we follow \cite{grt} and introduce
the flux map:
\begin{equation}
 \begin{array}{cccc}
  \delta: & H^1(KS_7) & \rightarrow & H^3(M) \\
          & [(\dot \om, \dot \phi, \dot \om)] & \mapsto & [ \dot T - 2c(\dot \theta, F)],
 \end{array}
\end{equation}
motivated by the flux quantization condition in Heterotic string theory. We note that
this map carries information on the infinitesimal variation of the Courant algebroid structure.
The kernel $\ker \delta$ of the flux map parametrizes infinitesimal variations of solutions
to (\ref{eq:Killing}) modulo (a subgroup of) symmetries $\Aut(E)$.
Then, following \cite{grt}, we obtain:

\begin{prop}
 \label{prop: link gen spinors, classical spinors}
 The group $H^1(\widehat{KS}_7)$ is given by an extension:
 \begin{equation}
  \label{eq: sequence gen H1}
  0 \rightarrow H^2(M) \rightarrow H^1(\widehat{KS}_7) \rightarrow \ker \delta \rightarrow 0
 \end{equation}
\end{prop}
The space $H^1(\widehat{KS}_7)$ is closer to the physical moduli space of $3$-dimensional compactifications of the heterotic string (preserving $N=1/2$ supersymmetry), as its elements are compatible with the flux quantization principle.


\begin{thebibliography}\frenchspacing\smallbreak

  \bibitem{Bryant} R.~Bryant, \emph{Some remarks on $G_2$-structures}
  	Proceedings of Gokova Geometry--Topology Conference 2005

  	\bibitem{BuCaGu} H. Bursztyn, G. Cavalcanti and M. Gualtieri,
  \emph{Reduction of Courant algebroids and generalized complex
    structures}, Adv. Math. {\bf 211} (2) (2007) 726--765.
  	
\bibitem{ChStXu} Z. Chen, M. Stienon and P. Xu, \emph{On regular
    Courant algebroids}, J. Symplectic Geom. {\bf 11} (2013) 1--24.

\bibitem{CSCW11} A. Coimbra, C. Strickland-Constable and D. Waldram, \emph{Supergravity as generalized geometry I: Type II theories}, arXiv:1107.1733v2.

\bibitem{OssaLaSv} X. de la Ossa, M. Larfors, and E. Svanes, {\it Exploring $SU(3)$ Structure Moduli Spaces with Integrable $G2$ Structures}, Adv. Theor. Math. Phys. {\bf 19} (2015) 837--903, arXiv:1409.7539.


\bibitem{DS} Simon Donaldson and Ed Segal, \emph{Gauge theory in higher dimensions, II}, Surveys in Differential Geometry, 16, (2011), 1-14.
    
\bibitem{DN} A.~Douglis and L. Niremberg, \emph{Interior estimates for elliptic systems of partial differential equations}, Comm. Pure App. Math. (4) {\bf 8} (1955) 503--538.    
    
    \bibitem{FerGr} M.~Fern\'andez and A.~Gray, \emph{Riemannian manifolds with structure group $G_2$} , Ann. Mat. Pura Appl. 32 (1982), 19--45.
   
\bibitem{FIUVa7} 
M. Fern\'andez, S. Ivanov, L. Ugarte, D. Vassilev, \emph{The quaternionic Heisenberg group and Heterotic String Solutions with non-constant dilaton in dimensions 7 and 5}, Comm. Math. Phys. {\bf 339} (2015), 199--219.    
   
\bibitem{FIUV7} 
M. Fern\'andez, S. Ivanov, L. Ugarte, R. Villacampa, \emph{Compact supersymmetric solutions of the heterotic equations of motion in dimensions 7 and 8}, Adv. Theor. Math. Phys. {\bf 15} (2011), 245--284.
    
\bibitem{FrIv03} T.~Friedrich and S.~Ivanov, \emph{Killing spinor equations in dimension 7 and geometry of integrable $G_2$ manifolds}, J. Geom. Phys. {\bf 48} (2003), 1--11.

\bibitem{Ferreira}
		A. C. Ferreira,
		\emph{A vanishing theorem in twisted De Rham cohomology}, arXiv:1012.2087 [math.DG] (2011).

\bibitem{HN} D. Harland and C. N\"olle, {Instantons and Killing spinors}, J. High Energy Phys., (2012), No. 3, 082.

\bibitem{GF} M. Garcia-Fernandez, \emph{Torsion-free generalized connections and
    heterotic supergravity}, Comm. Math. Phys. {\bf 332} (2014)
  89--115.

\bibitem{grt} M. Garcia-Fernandez, R. Rubio and C. Tipler, {\it Infinitesimal moduli for the Strominger system and generalized Killing spinors},
 arXiv:1503.07562 (2015).
 
\bibitem{GaMaWa} J. Gauntlett, D. Martelli, D. Waldram, \emph{Superstrings with Intrinsic Torsion}, Phys. Rev. {\bf D69} (2004) 086002.

\bibitem{GLL} J. Gray, M. Larfors, and D. L\"ust, {\it Heterotic domain wall solutions and SU(3) structure
manifolds}, JHEP {\bf 1208} (2012) 099, arXiv:1205.6208.

\bibitem{G3} M. Gualtieri, \emph{Branes on Poisson varieties}, \emph{The
    many facets of geometry}, 368--394, Oxford Univ. Press, Oxford,
  2010, arXiv:0710.2719.
  
\bibitem{GuNi} M. G\"unaydin, H. Nikolai, \emph{Seven-dimensional octonionic Yang-Mills instanton and its extension to an heterotic string soliton}, Phys. Lett. B {\bf 353} (1991), 169.

\bibitem{FrIv05} P.~Ivanov and S.~Ivanov, \emph{$SU(3)$-instantons and $G_2, Spin(7)$ heterotic string solitons}, Commun. Math. Phys. 259 (2005), 79--102.

\bibitem{Ivan09} S.~Ivanov, \emph{Heterotic supersymmetry, anomaly
    cancellation and equations of motion}, Phys. Lett. B {\bf 685}
  (2-3) (2010) 190--196.

\bibitem{Hit0} N. Hitchin, \emph{The geometry of three-forms in six and seven dimensions}, arXiv:0010054.

\bibitem{Hit1} \bysame, \emph{Generalized Calabi-Yau manifolds},
  Q. J. Math {\bf 54} (2003) 281--308.

\bibitem{joy0} D. Joyce, \emph{Compact Riemannian 7-Manifolds with Holonomy G2. I}, J. Diff. Geom. {\bf 43} (1996).

\bibitem{joy} \bysame, \emph{Compact Riemannian manifolds with special holonomy},
Oxford University Press, 2000.

\bibitem{Kar}
S. Karigiannis, \emph{Some Notes on $G_2$ and $Spin(7)$ Geometry},
Recent Advances in Geometric Analysis, Advanced Lectures in Mathematics, Vol. 11; International Press, (2010), 129--146.

\bibitem{ku}
    M. Kuranishi, \emph{New proof for the existence of locally
    complete families of complex structures}, in
    `Proc. Conf. Complex Analysis (Minneapolis 1964)',. Springer,
    Berlin, 1965, 142--154.

\bibitem{LiYau} J.~Li and S.-T.~Yau, \emph{The existence of supersymmetric string theory with torsion}, J. Diff. Geom. {\bf 70} (2005) 143--181.

\bibitem{LM1} R. B. Lockhart and R. C. Mc Owen, \emph{On elliptic systems
  in $\RR^n$}, Acta Mathematica (1) {\bf 150} (1983), 125-135.

\bibitem{LM2} \bysame, \emph{Elliptic differential operators on non-compact
  manifolds}, Annali della Scuola Normale Superiore di Pisa {\bf 12}
  (3) (1985) 409--447.
  
  \bibitem{O} G. Oliveira, \emph{Monopoles on the Bryant-Salamon $G_2$-manifolds}, J. Geom. Phys., 86, (2014), 599-632.

\bibitem{Rubioth} R. Rubio \emph{Generalized geometry of type $B_n$},
  Oxford University DPhil Thesis (2014).

\bibitem{SW} Henriqe S\'a Earp and Thomas Walpuski, {$G_2$-instantons on twisted connected sums}, Geom. Top. 19-1, (2015), 1263-1285.

  \bibitem{Strom} A.~Strominger, \emph{Superstrings with torsion},
  Nucl. Phys. B {\bf 274} (2) (1986) 253--284.
  
  \bibitem{TsengYau} L.-S.~Tseng and S.-T.~Yau, \emph{Non-K\"ahler
    Calabi-Yau Manifolds}, String-Math 2011, 241--254, Proc. Symposia in Pure Mathematics {\bf 85} (2012).
  
\bibitem{Wa} Th.~Walpuski, \emph{$G_2$-instantons on generalised Kummer constructions}, Geom. Topol. {\bf 17} (2013), no. 4, 2345-2388.

\bibitem{Wells} R.O.~Wells, \emph{Differential Analysis on Complex Manifolds}, Springer, 2008.

\end{thebibliography}
\end{document}